\documentclass[letterpaper, 10 pt, conference]{ieeeconf}
\usepackage{graphicx}
\usepackage{float}
\usepackage{amsfonts}
\usepackage{comment}
\usepackage{amsmath}
\usepackage{placeins}
\let\labelindent\relax
\usepackage{enumitem}
\usepackage{bm}
\usepackage{hyperref}
\usepackage{xcolor}
\usepackage[noadjust]{cite}
\usepackage{amssymb}
\usepackage{algorithmic}
\usepackage{textcomp}
\usepackage{cleveref}
\usepackage{lipsum}
\usepackage{titlesec}
\renewcommand{\baselinestretch}{.99}

\IEEEoverridecommandlockouts

\DeclareMathOperator{\sgn}{sgn}
\DeclareMathOperator{\tr}{tr}
\DeclareMathOperator{\vol}{vol}
\DeclareMathOperator{\spann}{span}
\DeclareMathOperator{\vect}{vec}
\DeclareMathOperator{\imageoffunction}{image}
\DeclareMathOperator{\convhull}{conv}
\DeclareMathOperator{\exposed}{exp}
\DeclareMathOperator{\closure}{cl}
\newcommand{\theset}[1]{\left\{ #1 \right\} }
\newcommand{\setst}{\;\vert\;}

\newcommand{\blambda}{\bm{\lambda}}  
\newcommand{\bxigreek}{\bm{\xi}}   
\newcommand{\bx}{\mathbf{x}}
\newcommand{\bxi}[1][i]{\mathbf{x}_{#1}}
\newcommand{\bu}{\mathbf{u}}
\newcommand{\bPhi}{\mathbf{\Phi}}

\newcommand{\bpsi}{\bm{\psi}}
\newcommand{\bolds}[1]{\mathbf{#1}}
\newcommand{\bA}{\mathbf{A}}
\newcommand{\bB}{\mathbf{B}}
\newcommand{\reachgram}[2][]{\mathbf{W}_{R}^{#1}(0,#2)}

\newcommand{\ol}[1]{\overline{#1}}
\newcommand{\ul}[1]{\underline{#1}}

\newcommand{\R}{\mathbb{R}}
\newcommand{\N}{\mathbb{N}}
\newcommand{\U}{\mathbb{U}}

\newcommand{\calU}{\mathcal{U}}
\newcommand{\calG}{\mathcal{G}}

\newcommand{\rd}{\mathop{}\!\mathrm{d}}

\newcommand{\pder}[2][]{\frac{\partial#1}{\partial#2}}
\newcommand{\Lp}[1][p]{\mathcal{L}^{#1}}

\newcommand{\Lpnorm}[2][p]{\| #2 \|_{\mathcal{L}^{#1}}}

\newcommand{\reachsets}[2][\calU_{\ul{\bu},\ol{\bu}}]{\mathcal{R}(#2;#1)}

\newcommand{\reachsetsExp}[3][\calU_{\ul{\bu},\ol{\bu}}]{\mathcal{R}^{#3}(#2;#1)}

\newtheorem{proposition}{Proposition}
\crefname{proposition}{Proposition}{Propositions}
\Crefname{proposition}{Proposition}{Propositions}

\newtheorem{definition}{Definition}
\newtheorem{theorem}{Theorem}
\crefname{theorem}{Theorem}{Theorems}
\Crefname{theorem}{Theorem}{Theorems}

\newtheorem{remark}{Remark}
\newtheorem{lemma}[theorem]{Lemma}
\crefname{lemma}{Lemma}{Lemmas} 
\Crefname{lemma}{Lemma}{Lemmas}

\newtheorem{example}{Example}
\crefname{example}{Example}{Examples}

\crefname{equation}{}{}    
\crefname{enumi}{}{}     
\crefname{figure}{Fig.}{Figs.}

\begin{document}

\bstctlcite{IEEEexample:BSTcontrol}

\title{Reachability Analysis for Design Optimization\\
\thanks{The authors would like to acknowledge the Multidisciplinary Science and Technology Center of the Aerospace Systems Directorate, Air Force Research Laboratory for funding and supporting this effort through the Collaborative Center for Design
and Research of Interdisciplinary Systems program.}
\thanks{Distribution A: Approved for public release; distribution unlimited.  Case no. AFRL-2025-4916.}
}

\author{Steven Nguyen \quad Jorge Cort\'es \quad Boris Kramer}

\maketitle

\begin{abstract}
    We present an approach to approximate reachable sets for linear systems with bounded $\mathcal{L}^\infty$ controls in finite time.
    Our first approach investigates the boundaries of these sets and reveals an exact characterization for single-input, planar systems with real, distinct eigenvalues.
    The second approach leverages convergence of the $\mathcal{L}^p$-norms to $\mathcal{L}^\infty$ and uses $\mathcal{L}^p$-norm reachable sets as an approximation of the $\mathcal{L}^\infty$-norm reachable sets.
    Our optimal control results yield insights that make computational approximations of the $\mathcal{L}^p$-norm reachable sets more tractable, and yield exact characterizations for $\mathcal{L}^\infty$ with the previous assumptions on the system.
    As an example, we incorporate our reachability analysis into the design optimization of a highly-maneuverable aircraft.
    Introducing constraints based on reachability allow us to factor physical limitations to desired flight maneuvers into the design process.
\end{abstract}

\maketitle

\section{Introduction}

\subsection{Control concepts in aircraft design}
In the design of next-generation aircraft, control considerations are critically important to ensuring that new air platforms have the capabilities to execute their maneuvers.
However, advanced modern control concepts are rarely considered during design optimization and are instead often reserved for later stages of analysis.
This is a common tradeoff that simplifies the design procedure at the expense of aircraft designs that are optimal for their missions.
We propose to incorporate reachability analysis into aircraft design optimization to enforce maneuverability and controlled capabilities of aircraft.
Metrics based on reachability enable a controller-agnostic investigation of aircraft capabilities that do not constrain subsequent construction of feedback controllers, making it an appealing approach for aircraft design.

\subsection{Background/literature review}
Interconnections between subsystems in design optimization have been studied since the late 20th century within the field of multidisciplinary design optimization (MDO)~\cite{martins_MultidisciplinaryDesignOptimization_2013}.
Since MDO was first proposed, aircraft systems have been popular subjects for study due to their interdisciplinary nature, coupling structural, aerodynamic, and control effects.
In the 1980s and 1990s, researchers applied MDO to aircraft wing optimization~\cite{livne_IntegratedMultidisciplinarySynthesis_1990,grossman_IntegratedAerodynamicStructural_1988}, and, with later advancements in computational power, began studying optimization of full aircraft designs~\cite{manning1999large}.
Furthermore, control co-design, the procedure of optimizing a structure with control metrics in consideration~\cite{garcia-sanz_ControlCoDesignEngineering_2019}, has been known to yield improvements in controlled performance~\cite{skelton_OptimalMixStructure_1992,gulewicz_SetBasedRobustControl_2025}.

In the field of aircraft design, considerations of control within MDO vary widely~\cite{bahiamonteiro_ControlRelatedMetrics_2024}.
Some approaches fix the control architecture and optimize a controller simultaneously with design~\cite{livne_IntegratedMultidisciplinarySynthesis_1990,khot_MulticriteriaOptimizationDesign_1998}.
Although this approach simplifies the choice of design metric, the engineer must decide early in the design stage what type of controller will be used on the final system, risking suboptimality if the control scheme changes later.
Alternatively, control-agnostic metrics for design optimization have been explored recently, ranging from frequency-domain metrics for robustness~\cite{bahiamonteiro_DesignMetricsLanding_2024} to state-space metrics based on optimal control~\cite{gupta_ControllabilityGramianControl_2020,cunis_IntegratingNonlinearControllability_2023}.
These methods have the advantage of characterizing the behavior of the aircraft with fewer assumptions on the control synthesis technique, but are more complex to formulate.
We build on these investigations of control-metrics for MDO by proposing the use of metrics based on reachability.

Reachability analysis studies the set of states a dynamical system can reach, and is often used for safety verification.
Many approaches for studying reachability use set propagation techniques to compute under- and over-approximations of reachable sets of linear, nonlinear, and hybrid systems~\cite{althoff_SetPropagationTechniques_2021,maler_ComputingReachableSets_2008}.
For linear systems, convexity of reachable sets make certain classes of subsets of Euclidean space particularly desirable to use. 
Depending on desired properties such as closure under linear transformations, Minkowski sums, or intersections, engineers may consider using ellipses~\cite{kurzhanski_EllipsoidalTechniquesReachability_2000} or polytopes~\cite{chutinan_ComputingPolyhedralApproximations_1998}.
Numerous computational packages exist to handle these computations, like SpaceEx~\cite{frehse_SpaceExScalableVerification_2011} and CORA~\cite{althoff_IntroductionCORA2015_2015}.
In the nonlinear setting, reachable sets are rarely convex and formal analysis typically requires solving a Hamilton-Jacobi partial differential equation~\cite{bansal_HamiltonJacobiReachabilityBrief_2017}.
To circumvent the computational complexity of nonlinear reachability, many approximation approaches exist, like hybrid zonotopic~\cite{siefert_ReachabilityAnalysisUsing_2025}, learning~\cite{bansal_DeepReachDeepLearning_2021}, and sampling~\cite{lew_ConvexHullsReachable_2025} methods.
Reachability analysis has been useful for applications like quadcopter control~\cite{gillula_DesignGuaranteedSafe_2010}, robotic manipulators~\cite{porges_ReachabilityCapabilityAnalysis_2014}, and aircraft auto-landing~\cite{bayen_AircraftAutolanderSafety_2007}, but has largely been unexplored in design optimization.

\subsection{Contributions}
In this paper, we study reachable sets of linear systems for design optimization of aircraft.
By avoiding the computational complexity of nonlinear reachability, control metrics constructed based on linear reachability are easier to evaluate at each iteration of the optimization problem.
To understand the capabilities of aircraft in realistic scenarios, we study characterizations and approximations of reachable sets under bounded-magnitude inputs.
Although ellipsoidal and zonotopic approaches can approximate these sets with great accuracy, we seek to design constraint metrics based on the system matrices that allow us to evaluate properties of the reachable set.
Our main theoretical contributions are a characterization of the boundary of $\Lp[\infty]$ reachable sets for a subset of planar linear systems with single inputs and an investigation of $\Lp$-norm reachable sets as approximations of the $\Lp[\infty]$-norm reachable sets for general linear systems.
We first note that control inputs driving a system to the exposed points of its reachable set must be saturated at all times, and leverage this to construct a parametric function that describes the boundary of certain systems' reachable sets.
Using calculus of variations, we show a derivation of $\Lp$-norm optimal controls for linear systems and present a numerical approach to approximating $\Lp$-norm reachable sets.
Lastly, we incorporate these developments into design optimization to optimize aircraft maneuverability around points of interest within the flight envelope.

\section{Preliminaries}

This section presents our notational conventions and preliminaries on reachable sets of linear systems.

\subsection{Notation}
Throughout the paper, we use bold-faced symbols to denote vector-valued and matrix-valued quantities.
We denote the Hadamard product by $\odot$.
For vectors $\bx, \bolds{y} \in \R^n$, $\bx~\odot~\bolds{y}~=~[x_1 y_1,\; \ldots \;, x_n y_n ]^\top \in \R^n$.
For matrices $\bA \in \R^{n \times m}, \vect(\bA)$ refers to the vectorization of $\bA$, attained by stacking its columns.
For $\alpha \in \R$, $\bx^\alpha$ denotes the element-wise power of $\bx$.
We use $\langle\cdot,\cdot \rangle: \R^n \times \R^n \to \R$ to denote the inner product on Euclidean space, $\langle \bx,\bolds{y}\rangle = \bx^\top \bolds{y}$, where $^\top$ denotes the transpose of a vector or matrix.
For $p \in [1,\infty)$, we define $\|\bx\|_p^p  := \sum_{i=1}^n |x_i|^p$ and for $p = \infty$, $\|\bx\|_\infty := \max_{i=1,\ldots,n} |x_i|$.
Generalizing to norms on signals, given $\bolds{f}:[0,T] \to \R^m$, we define the $p$th power of the $\Lp$ norm as $\Lpnorm{\bolds{f}(t)}^p := \int_0^T \|\bolds{f}(t)\|_p^p \rd t$.
Similarly, the $\Lp[\infty]$-norm is defined as $\Lpnorm[\infty]{\bolds{f}(\cdot)} = \max_{1 \leq i \leq n, t \in [0,T]} |f_i(t)|$.
Then, the function space $\Lp(\mathcal{I})$ is the set of functions defined over set $\mathcal{I}$ with a finite $\Lp$-norm.
The image of a function is denoted as $\imageoffunction(\bolds{f}) = \theset{ \bolds{f}(t) \setst t \in [0,T]}$.
The convex hull, exposed points, and closure of a set $E\subseteq \R^n$ are denoted by $\convhull(E)$, $\exposed(E)$, and $\closure(E)$.

\subsection{Reachable sets of linear systems}

In this paper, we consider systems of the form
\begin{equation}
    \dot{\bx}(t) = \bA\bx(t) + \bB\bu(t), \label{eq:basic_linear_system}
\end{equation}
where the state $\bx(t) \in \R^n$, $\bu(t) \in \R^m$, and $\bA$, $\bB$ have appropriate dimensions.
We denote the solution at time $t$ of \cref{eq:basic_linear_system} starting from $\bxi[0]$, driven by control signal $\bu(\cdot):[0,T] \to \R^m$ by $\bPhi(t,\bxi[0],\bu(\cdot))$.

Without loss of generality, we consider reachable sets from the origin because nonzero initial conditions will simply shift the reachable set by $e^{\bA T}\bx(0)$.
Given a set of admissible control inputs $\calU \subseteq \R^m$ convex and compact, we denote the reachable set at time $T$ of \cref{eq:basic_linear_system} by
\begin{align}\nonumber
    \reachsets[\calU]{T} := \{ \bPhi(T,\bxi[0],\bu(\cdot)) \setst & \bx(0) = \bolds{0},\; \bu(t) \in \calU \\ & \forall t \in [0,T] \}.  
\end{align}
We also consider admissible controls from the set of measurable functions $\U := \theset{\bu(t):[0,T] \to \R^m }$ satisfying an $\Lp$-norm bound and define $\U_{\Lp} := \left\{\bu(\cdot) \in \R^m \setst \Lpnorm{\bu(t)} \leq c  \right\}$ for some $p \in \N \cup \{\infty\}$ and $c \in \R_{>0}$.
With an abuse of notation, we may also denote the set of reachable states with the $\Lp$-norm bound on the input signal as $\reachsets[\U_{\Lp}]{T}$.

In the special case where $p=2$, it is known~\cite{kalman_ContributionsTheoryOptimal_1960} that $\reachsets[\U_{\Lp[2]}]{T}$ is an ellipsoid characterized by the reachability Gramian $\reachgram{T} = \int_0^T e^{\bA(T-\tau)}\bB\bB^\top e^{\bA^\top (T-\tau)} \rd \tau$ of \cref{eq:basic_linear_system}.
In particular, the $i$th principal axis of $\reachsets[\U_{\Lp[2]}]{T}$ has length $\sqrt{c\lambda_i}$ in the direction of $\bolds{v}_i$, where $(\lambda_i,\bolds{v}_i)$ are an eigenvalue/ eigenvector pair of $\reachgram{T}$.

\section{Reachable set boundaries for \texorpdfstring{$p=\infty$}{p=inf}}
To synthesize constraints for optimization from reachability analysis, we seek algebraic characterizations of reachable sets that can be evaluated based on the $\bA$ and $\bB$ matrices of \cref{eq:basic_linear_system}.
Similarly to how the eigenvectors/eigenvalues of $\reachgram{T}$ describe the geometry of the $\Lp[2]$-norm reachable set, we seek an algebraic representation for the $\Lp[\infty]$ case.

We present a parametric function representation of the reachable set boundary for planar LTI systems with single inputs and real, distinct eigenvalues.
For these systems, our results yield an exact characterization of the set boundary and, as a consequence of convexity, the reachable set itself.

\subsection{Optimal switching controllers}
The first step in this characterization is to recognize that due to convexity and compactness of reachable sets for LTI systems, the boundaries of reachable sets are characterized by the set's exposed points.
Recall that a point on the boundary of a convex set is an exposed point if it strictly maximizes some linear functional~\cite{rockafellar_ConvexAnalysis_2015}.
In the following lemma, we show that these exposed points are achieved by controls on the boundary of the admissible control set.

\begin{lemma} \label[lemma]{lem:saturated_controls_on_boundary}
    For linear system \cref{eq:basic_linear_system} with $\calU_{\ul{\bu},\ol{\bu}}=\theset{ \bu \in \R^m \setst \ul{\bu} \leq \bu \leq \ol{\bu} }$, where $\ul{\bu},\ol{\bu} \in \R^m$ are constant vectors, the exposed points of $\reachsets{T}$ are achieved by control signals of the form
    \begin{equation}~\label{eq:fully_saturated_control}
        u_i(t) = \left\{ \begin{matrix}
            \ol{u}_i, & \sgn(\psi_i(t;\bolds{c})) \geq 0 \\
            \ul{u}_i, & \sgn(\psi_i(t;\bolds{c})) < 0 ,
        \end{matrix} \right. 
    \end{equation}
    where $\bolds{\psi}_i(t;\bolds{c}) = \bolds{c}^\top e^{\bA(T-t)}\bolds{b}_i$, $\bolds{b}_i$ is the i-th column of $\bB$, and $\bolds{c} \in \R^n$ is a constant vector.
\end{lemma}

The proof of this result is omitted and will appear elsewhere.
In addition to providing the form of controls that steer the system to the exposed points of its reachable set, Lemma~\ref{lem:saturated_controls_on_boundary} also implies that those controls are saturated at all times.

\begin{remark}
    The control signal~\eqref{eq:fully_saturated_control} matches the $\mathcal{L}^\infty$-optimal control discussed in~\cite{pecsvaradi_ReachableSetsLinear_1971} if $\ol{\bu} = -\ul{\bu}$.
However, we consider a more general magnitude bound that can be asymmetric.
\end{remark}

Furthermore, for the specific case of planar systems with real and distinct eigenvalues, the number of times the components of the control signal can switch between their maximum and minimum values is upper bounded by one.

\begin{proposition} \label{prop:psi_switches_once}
    Given \cref{eq:basic_linear_system}, assume $\bA$ is diagonalizable with real, distinct eigenvalues, and that $\bA$ is of dimension 2.
    Consider $\bpsi(t;\bolds{c}) = \bolds{c}^\top e^{\bA(T-t)}\bB$ as defined in \cref{lem:saturated_controls_on_boundary} for any arbitrary $\bolds{c} \in \R^n$.
    Then, for any component $\psi_i(t;\bolds{c})$, only one of the two following statements is true:
    \begin{enumerate}
        \item $\psi_i(t;\bolds{c}) = 0$ $\forall t \in [0,T]$.
        \item There exists at most one value of $\eta \in [0,T]$ such that $\psi_i(\eta;\bolds{c}) = 0$. \label{condition:only_one_switch}
    \end{enumerate}

\end{proposition}

\begin{proof}
    Let $\bA = \bolds{VDV}^{-1}$ be the diagonalization of $\bA$ where the columns of $\bolds{V}$ are eigenvectors and $\bolds{D}$ is the diagonal matrix of eigenvalues.
    Since $\spann (\bolds{V}) = \R^2$, $\exists \alpha_{1,i},\alpha_{2,i}$ such that $\bolds{b}_i = \alpha_{1,i}\bolds{v}_1 + \alpha_{2,i}\bolds{v}_2$.
    Substituting this into the expression of $\psi_i(t;\bolds{c})$ and denoting $\tilde{\bolds{c}}^\top = \bolds{c}^\top e^{\bA T}$, we find $\psi_i(t;\bolds{c}) = \alpha_{1,i}e^{-\lambda_1 t} \langle \tilde{\bolds{c}},\bolds{v}_1 \rangle + \alpha_{2,i} e^{-\lambda_2 t} \langle \tilde{\bolds{c}},\bolds{v}_2 \rangle$.
    Now, suppose that $\psi_i(\eta_1;\bolds{c}) = 0 = \psi_i(\eta_2;\bolds{c})$ for $\eta_1 \neq \eta_2$.
Expanding these expressions yields
    \begin{align}
        \alpha_{1,i} h_1 e^{-\lambda_1 \eta_1} + \alpha_{2,i}h_2 e^{-\lambda_2 \eta_1} &= 0 \label{eq:psi_0_proof_upper}\\ \label{eq:psi_0_proof_lower}
        \alpha_{1,i} h_1 e^{-\lambda_1 \eta_2} + \alpha_{2,i}h_2 e^{-\lambda_2 \eta_2} &= 0,
    \end{align}
    where $h_i = \langle \tilde{\bolds{c}},\bolds{v}_i \rangle$.
    Suppose w.l.o.g that $\eta_1 = \eta_2 + \epsilon$ for some $\epsilon > 0$.
    Then, we can substitute into \cref{eq:psi_0_proof_upper} to find
    \begin{align} \label{eq:psi_0_proof_intermediary}
        \alpha_{1,i}h_1e^{-\lambda_1 \eta_2}e^{-\lambda_1 \epsilon} + \alpha_{2,i}h_2 e^{-\lambda_2 \eta_2}e^{-\lambda_2 \epsilon} = 0.
    \end{align}
    Then, substituting for $\alpha_{1,i}$ in \cref{eq:psi_0_proof_intermediary} using \cref{eq:psi_0_proof_lower}, we can further simplify to
    \begin{equation}  \label{eq:psi_0_first_cond}
        \alpha_{2,i}h_2 e^{-\lambda_2 \eta_2}\left( e^{-\lambda_2 \epsilon} - e^{-\lambda_1 \epsilon} \right) = 0.
    \end{equation}
    Applying the same argument, but substituting instead for the $\alpha_{2,i}$ term, we find
    \begin{equation}
        \alpha_{1,i}h_1e^{-\lambda_1 \eta_2} \left( e^{-\lambda_1 \epsilon} - e^{-\lambda_2 \epsilon} \right) = 0 \label{eq:psi_0_second_cond}.
    \end{equation}
    From the assumptions on $\lambda_i$ being distinct, \cref{eq:psi_0_first_cond} is satisfied if either of $\alpha_{2,i} = 0$ or $h_2=0$.
    Suppose $\alpha_{2,i} = 0$.
    Then, $\alpha_{1,i} \neq 0$ for nontrivial $\bolds{b}_i$ and $h_1 = \langle \bolds{\tilde{c}},\bolds{v}_1 \rangle = 0$ to satisfy \cref{eq:psi_0_second_cond}.
    Then, by definition, $\psi_i(t;\bolds{c}) \equiv 0$.
    The case of $h_2 = 0$ leads to the same result.
    In the case where all of $\alpha_{1,i},\alpha_{2,i},h_1,h_2$ are nonzero, then \cref{eq:psi_0_first_cond,eq:psi_0_second_cond} cannot be satisfied and we have a contradiction that implies $\epsilon = 0$ and $\eta_1 = \eta_2$.
    Thus, either $\psi_i(t;\bolds{c})~\equiv~0$ or there is at most one $\eta \in [0,T]$ satisfying $\psi_i(\eta;\bolds{c}) = 0$.

\end{proof}

\begin{remark}
    Switching controllers are also discussed in~\cite[Thms. 9, 10]{pontryagin_MathematicalTheoryOptimal_2018}.
    Here, we expand on these results to provide the form of the controller~\eqref{eq:fully_saturated_control}, parameterized by the vector~$\bolds{c}$.
    \end{remark}

\Cref{prop:psi_switches_once} states that under the given assumptions, any component of the optimal control driving the system to an exposed point of $\reachsets{T}$ only switches between its maximum and minimum values at most one time.
The intuition for the restrictive assumptions is that the exponential matrix in $\bpsi(t;\bolds{c})$ connects the evolution of $\bpsi(t;\bolds{c})$ to the evolution of \cref{eq:basic_linear_system}.
Thus, the number of times $\bpsi(t;\bolds{c})$ flips sign can increase for higher-dimensional systems or for oscillatory systems.
However, when these assumptions hold, a convenient parameterization can be uncovered.

\subsection{Parameteric function representation of reachable sets}
 We next use \Cref{lem:saturated_controls_on_boundary} and \cref{prop:psi_switches_once} to characterize $\reachsets{T}$ using the images of parametric functions.
\begin{theorem} \label{thm:parameterizing_singleinput_2dim_boundary}
    Using the assumptions in \cref{prop:psi_switches_once}, assume further that there is only one input to the system, bounded by $\ol{u},\ul{u} \in \R$.
    For $\eta \in [0,T]$, consider the functions:
    \begin{align} \label{eq:boundary_eq1}
        \bxigreek_1(\eta) &= \ol{u}\int\limits_0^\eta e^{\bA(T-\tau)}\bB \rd \tau + \ul{u}\int\limits_\eta^T e^{\bA(T-\tau)}\bB \rd \tau \\ \label{eq:boundary_eq2}
        \bxigreek_2(\eta) &= \ul{u}\int\limits_0^\eta e^{\bA(T-\tau)}\bB \rd \tau + \ol{u}\int\limits_\eta^T e^{\bA(T-\tau)}\bB \rd \tau
    \end{align}
    as mapping from $[0,T] \to \R^n$ and let $\calG_1 = \imageoffunction(\bxigreek_1), \, \calG_2 = \imageoffunction(\bxigreek_2)$.
    Then, $\reachsets[\calU_{\ul{u},\ol{u}}]{T} = \convhull\left(\calG_1 \bigcup \calG_2 \right)$.
\end{theorem}
\begin{proof}
    First, we show $\convhull \left(\calG_1 \bigcup \calG_2 \right) \subseteq \reachsets[\calU_{\ul{u},\ol{u}}]{T}$.
    By definition, if $\bolds{g} \in \calG_1$, there exists some $\hat{\eta}$ such that $\bolds{g} = \bxigreek_1(\hat{\eta})$, implying that $\bolds{g} = \bPhi(T,0,\tilde{u})$, where
    \begin{equation*}
            \tilde{u}(t) = \left\{ \begin{matrix} \ol{u}, & t \in [0,\hat{\eta}] \\ \ul{u}, & t \in (\hat{\eta},T]. \end{matrix}  \right.
    \end{equation*}
    The same holds if $\bolds{g} \in \calG_2$ where $\tilde{u}(t)$ has $\ol{u},\ul{u}$ swapped.
    This trivially shows that $\calG_1 \bigcup \calG_2 \subseteq \reachsets[\calU_{\ul{u},\ol{u}}]{T}$.
    Then, by convexity of $\reachsets[\calU_{\ul{u},\ol{u}}]{T}$, we have that $\convhull \left(\calG_1 \bigcup \calG_2 \right) \subseteq \reachsets[\calU_{\ul{u},\ol{u}}]{T}$.

    To see the reverse inclusion, note that by \cref{lem:saturated_controls_on_boundary} and \cref{prop:psi_switches_once}, the exposed points of $\reachsets[\calU_{\ul{u},\ol{u}}]{T}$ are achieved by controls that are saturated at all times and switch at most once.
    Since $\calG_1$ and $\calG_2$ capture all solutions of the LTI system using saturated controllers with at most one switch, $\exposed \left( \reachsets[\calU_{\ul{u},\ol{u}}]{T}\right) \subseteq \convhull \left( \calG_1 \bigcup \calG_2 \right)$.
    By convexity and compactness of the reachable set and the compactness of $\calG_1$ and $\calG_2$, $\reachsets[\calU_{\ul{u},\ol{u}}]{T} = \closure \convhull \exposed \left( \reachsets[\calU_{\ul{u},\ol{u}}]{T} \right) \subseteq \convhull \left( \calG_1 \bigcup \calG_2 \right)$~\cite{rockafellar_ConvexAnalysis_2015}.

\end{proof}

\Cref{thm:parameterizing_singleinput_2dim_boundary} uses the solution to linear systems to compute the trajectory at time $T$ for all controls that take the form of \cref{eq:fully_saturated_control}.
By leveraging the knowledge that these controls will hold their maximum or minimum values and switch at some point in the time window $[0,T]$, \cref{eq:boundary_eq1,eq:boundary_eq2} captures $\exposed \reachsets[\calU_{\ul{u},\ol{u}}]{T}$ and their convex hull reveals $\reachsets[\calU_{\ul{u},\ol{u}}]{T}$.
We next show an application of \cref{thm:parameterizing_singleinput_2dim_boundary}.

\begin{example} \label{example:reach_set_boundary}
    Consider the system:
\begin{equation} \label{eq:example_2d_system}
    \dot{\bx}(t) = \begin{bmatrix}
        0.4 & -0.3 \\ 0.5 & 1.7
    \end{bmatrix}\bx(t) + \begin{bmatrix}
        1 \\0
    \end{bmatrix} u(t).
\end{equation}
The $\bA$ matrix of \cref{eq:example_2d_system} is 2-dimensional and has real eigenvalues 0.53 and 1.57, so it satisfies the necessary assumptions for \cref{prop:psi_switches_once}
The systems is also single input, so \cref{thm:parameterizing_singleinput_2dim_boundary} can be applied.
Considering inputs of the set $\calU_{-1,1} = \theset{u \setst -1 \leq u \leq 1}$ and substituting $\ul{u} = -1$ and $\ol{u} = 1$ into \cref{eq:boundary_eq1,eq:boundary_eq2}, we plot $\calG_1$ and $\calG_2$ against a zonotopic inner-approximation of the reachable set using CORA~\cite{althoff_IntroductionCORA2015_2015} as shown in \cref{fig:boundary_parameterization_example}.
An exact characterization of $\reachsets[\calU_{-1,1}]{1}$ is given by \cref{eq:boundary_eq1,eq:boundary_eq2}.
As seen in \cref{fig:boundary_parameterization_example}, the parameterization can fully capture the boundary of the reachable set for this system.
Taking the convex hull of the parameterization or taking convex combinations of points in $\calG_1$ or $\calG_2$ will yield $\reachsets[\calU_{-1,1}]{1}$.

\begin{figure}[!t]
    \centering
    \includegraphics[width=0.8\linewidth]{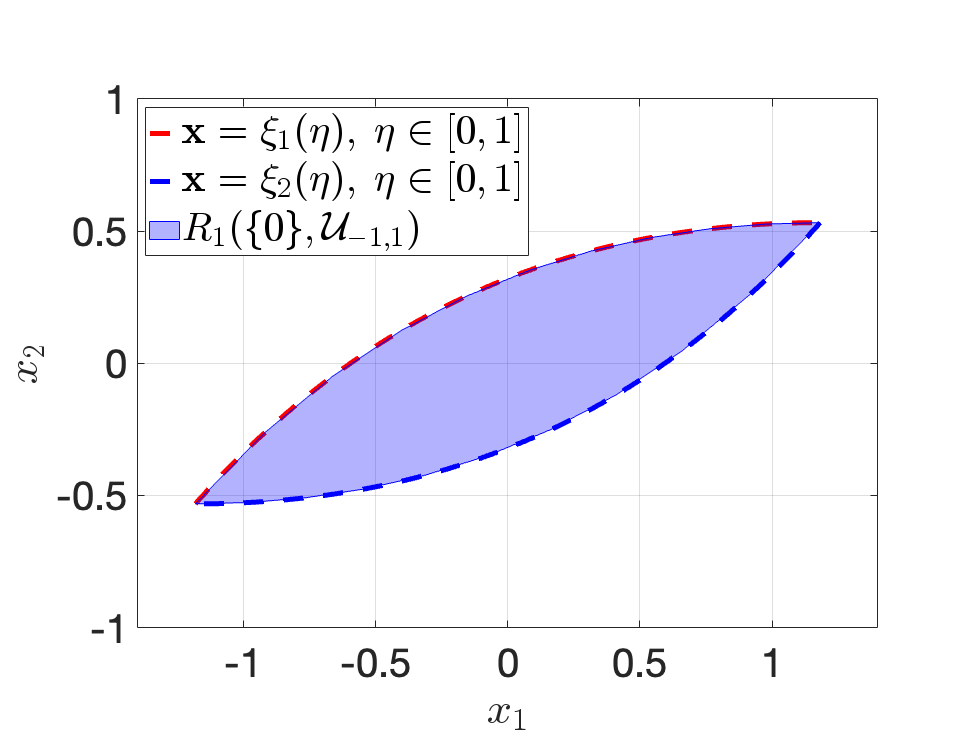}
    \caption{Comparison of the boundary parameterization using the switching control scheme as outlined in \cref{thm:parameterizing_singleinput_2dim_boundary} with a zonotopic inner-approximation of the reachable set.}
    \label{fig:boundary_parameterization_example}
\end{figure}

\end{example}

Although this approach provides an exact characterization of the reachable set boundary, the assumptions are restrictive, and the results do not scale to higher dimensions or multiple inputs.
In the following section, we study reachable sets for general linear systems and instead use $\Lp$-norms to approximate $\Lp[\infty]$ magnitude bounds.

\section{Reachable sets for finite \texorpdfstring{$p$}{p}}\label{sec:p_norm_reach_sets}

In this section, we remove restrictions on the system matrices and number of inputs, and consider reachable sets using inputs with bounded $\Lp$-norm, i.e. $\reachsets[\U_{\Lp}]{T}$, where $\U_{\Lp} = \theset{\bu(t) \in \U \setst \Lpnorm{\bu(t)} \leq 1}$.

\subsection{\texorpdfstring{$p$}{p}-norm optimal controls}
To understand the set of states that can be reached using controls with a bounded $\mathcal{L}^p$ norm, we must first consider controls that are optimal in the $\mathcal{L}^p$ norm.
This problem is discussed in~\cite{pecsvaradi_ReachableSetsLinear_1971} without proof, which we provide here.
The finite-time optimal control problem we seek to solve is
\begin{align} \nonumber
    \min_{\bu(\cdot) \in \mathcal{L}^p([0,T];\R^m)} \quad & \int\limits_0^T L(\bx(\tau),\bu(\tau)) \rd \tau \\  \label{eq:opt_contr_prob}
    s.t. \quad & \dot{\bx}(t) = \bolds{f}(\bx(t),\bu(t)) \\ \nonumber 
    & \bx(0) = \bolds{0}, \quad \bx(T) = \bxi[f]
\end{align}
where $L(\bx(\tau),\bu(\tau)) = \frac{1}{p}\|\bu(\tau)\|^p_p$, $\bolds{f}(\bx(t),\bu(t)) = \bA\bx(t) + \bB\bu(t)$, $\bxi[f]$ is the desired final state of the linear system at time $T$, and $\bolds{0}$ is an appropriately-sized vector of zeros.
First, we introduce the notion of a Hamiltonian.
\begin{definition}[\cite{athans2007optimal}]
     Let $H(\bx(t),\blambda(t),\bu(t)) = L(\bx(t),\bu(t)) + \langle\blambda(t),\bA\bx(t)+\bB\bu(t)\rangle$ where $\bx(t) \in \R^n$, $\blambda(t) \in \R^n$, and $\bu(t) \in \R^m$.
     We call $H(\bx(t),\blambda(t),\bu(t))$ the Hamiltonian of our problem and $\blambda(t)$ the costate vector.
\end{definition}
Applying Pontryagin's Minimum Principle, we can derive the $\mathcal{L}^p$-norm optimal controls, which we define as the solutions to the optimal control problem \cref{eq:opt_contr_prob}.
\begin{theorem} \label{thm:p-norm_control}
    The optimal control for \cref{eq:opt_contr_prob} is \begin{equation}
        \bu^*(t) = \left| -\bB^\top e^{-\bA^\top t} \blambda^*(0) \right|^{\frac{1}{p-1}}\odot \, \sgn(-\bB^\top e^{-\bA^\top t} \blambda^*(0)) \label{eq:opt_control_theorem}
    \end{equation}where $\blambda^*(0)$ is the costate initial condition and satisfies 
    $\bxi[f] = \int\limits^T_0 e^{\bA(t-\tau)} \bB \bu^*(t)$.
    In the case where $p$ is even, this expression can be simplified to
    \begin{equation} \label{eq:opt_p_control_even}
        \bu^*(t) = \left(-\bB^\top e^{-\bA^\top t} \blambda^*(0)\right)^{\frac{1}{p-1}}
    \end{equation}
\end{theorem}
\begin{proof}
    Following the application of Pontryagin's Minimum Principle to the linear fixed-end-point, fixed-time problem in~\cite{athans2007optimal}, we first form the Hamiltonian
    \begin{align} \label{eq:hamiltonian_def}
        H(\bx(t),\blambda(t),\bu(t)) = & \frac{1}{p}\|\bu(t)\|^p_p \\ &+ \blambda(t)^\top \bA\bx(t) + \blambda(t)^\top \bB\bu(t). \nonumber
    \end{align}
    Then, for an optimal control $\bu^*(t)$ and corresponding optimal trajectory $\bx^*(t)$, there exists $\blambda^*(t)$ such that:
    \begin{enumerate}[label=\alph*)] 
        \item The Hamilton canonical equations are satisfied:\label{PMP_hamilton_canonical}
        \begin{align*}
            \dot{\bx}^*(t) &= \frac{\partial H}{\partial \blambda} (\bx^*(t),\blambda^*(t),\bu^*(t)) = \bA\bx^*(t) + \bB\bu^*(t) \\
            \dot{\blambda}^*(t) &= -\frac{\partial H}{\partial \bx}(\bx^*(t),\blambda^*(t),\bu^*(t)) = -\bA^\top \blambda^*(t)
        \end{align*}
        \item $\bu^*(t)$ minimizes the Hamiltonian: \label{PMP_hamiltonian_min}
        \begin{equation} \label{eq:hamiltonian_min}
            H(\bx^*(t),\blambda^*(t),\bu^*(t)) \leq H(\bx^*(t),\blambda^*(t),\bu(t))
        \end{equation}$\forall \bu \in \R^m,\, \forall t \in \R_{+}$.
    \end{enumerate}
    To verify condition \cref{PMP_hamiltonian_min}, \cref{eq:hamiltonian_min} can be written as:
    \begin{equation*}
        \frac{1}{p}\|\bu^*(t)\|^p_p + \blambda^*(t)^\top \bB\bu^*(t) \leq \frac{1}{p}\|\bu(t)\|^p_p + \blambda^*(t)^\top \bB \bu(t)
    \end{equation*}
    Letting $\psi(\bu(t)) = \frac{1}{p}\|\bu(t)\|^p_p + \blambda^*(t)^\top \bB\bu(t)$, $\bu^*(t)$ should achieve the global minimum of $\psi(\bu(t))$ for all values of $t$.
    The first-order necessary condition for this is:
    \begin{align*}
        \left.\frac{\partial \psi(\bu(t))}{\partial \bu(t)}\right\vert_{\bu(t) = \bu^*(t)} &= \bolds{0}_{1\times m} \\
        \pder{\bu(t)}\left( \frac{1}{p}\|\bu(t)\|^p_p \right) + \blambda^*(t)^\top \bB &= \bolds{0}_{1\times m} \\
        p \frac{1}{p}\|\bu(t)\|_p^{p-1} \left(\frac{\bu(t) \odot |\bu(t)|^{p-2}}{\|\bu(t)\|_p^{p-1}} \right)^\top + \blambda^*(t)^\top \bB &= \bolds{0}_{1\times m} \\
        \bu(t) \odot|\bu(t)|^{p-2} + \bB^\top \blambda^*(t) &= \bolds{0}_{1 \times m}
    \end{align*}
    where the absolute value of a vector is applied component-wise.
    It can be verified that the candidate $\bar{\bu}(t) = \left| -\bB^\top \blambda^*(t) \right|^{\frac{1}{p-1}}\odot \, \sgn \left(-\bB^\top \blambda^*(t) \right)$ satisfies the above equality.
    Checking the second-order necessary conditions for a global minimum of $\psi(\bu(t))$, we can see that
    \begin{equation*}
        \left.\frac{\partial^2 \psi(\bu(t))}{\partial \bu ^2} \right|_{\bu(t) = \bar{\bu}(t)} =
        \begin{bmatrix}
        d_1 &        & 0 \\
            & \ddots &   \\
        0   &        & d_m
        \end{bmatrix}
    \end{equation*} 
    where $d_i = |\bar{u}_i|^{p-2} + (p-2)\,\bar{u}_i\,|\bar{u}_i|^{p-3}\,\sgn(\bar{u}_i)$ and the dependence of components of $\bar{\bu}(t)$ on time and the costate are omitted for brevity.
    The diagonal elements are all nonnegative, so $\frac{\partial^2}{\partial\bu^2} \psi(\bu(t))$ is positive semidefinite.
    Thus, $\bar{\bu}(t)$ minimizes $\psi(\bu(t))$ and satisfies condition~\cref{PMP_hamiltonian_min}.

    Referring to the Hamilton canonical equations in condition~\cref{PMP_hamilton_canonical}, it is clear that $\blambda^*(t) = e^{-\bA^\top t}\blambda^*(0)$, which can be substituted into $\bar{\bu}(t)$.
    Enforcing the terminal state with the canonical equations also imposes that $\bxi[f] = \int^T_0 e^{\bA(t-\tau)} \bB \bu^*(t)$.
    Thus, $\bar{\bu}(t)$ minimizes \cref{eq:opt_contr_prob}.
\end{proof}

In the case where $p=2$, this result can be shown to be identical to the case of minimum-energy control~\cite{kalman_ContributionsTheoryOptimal_1960}:
\begin{equation} \nonumber
    \bu^*(t) = \bB^\top e^{\bA^\top (T-t)} \reachgram[-1]{T} \bxi[f].
\end{equation} However, for the more general case of $p > 2$, we have not found an analytical method to solve for the costate initial condition $\blambda^*(0)$ in terms of the terminal state $\bxi[f]$.
So even though we know the form of the optimal control to reach any point from the origin, evaluating that control for a specific point is not straightforward.
While this poses challenges for an exact characterization of $p$-norm reachable sets, \cref{thm:p-norm_control} opens the door to numerical approximations of $p$-norm reachable sets.
Rather than searching across the entire \textit{function space} of possible control inputs $\U_{\Lp}$, this reduces the approximation of $\Lp$-norm reachable sets to searching over a grid of costate initial conditions $\blambda_0 \in \R^n$ in a \textit{vector space}. 

\begin{example} \label{example:lp_norm_reach_set}

Consider the system \cref{eq:example_2d_system} from \cref{example:reach_set_boundary}.
Now, we present a numerical approximation of $\reachsets[\U_{\Lp[\infty]}]{1}$ using $\reachsets[\U_{\Lp[6]}]{1}$.
Sampling 1,525 values of $\blambda_0$ ranging between $[\pm 5,\pm 100]^\top$ and simulating trajectories from the origin using controllers of the form \cref{eq:opt_control_theorem} establishes an association between terminal points $\bxi[f]$ and the optimal $\Lp[6]$-norm controllers to reach them.
Then, by computing the cost of each controller, we can determine which states were reachable using inputs with unit $\Lp[6]$-norm cost in time $T =1$.

\begin{figure}[!t]
    \centering
    \includegraphics[width=0.8\linewidth]{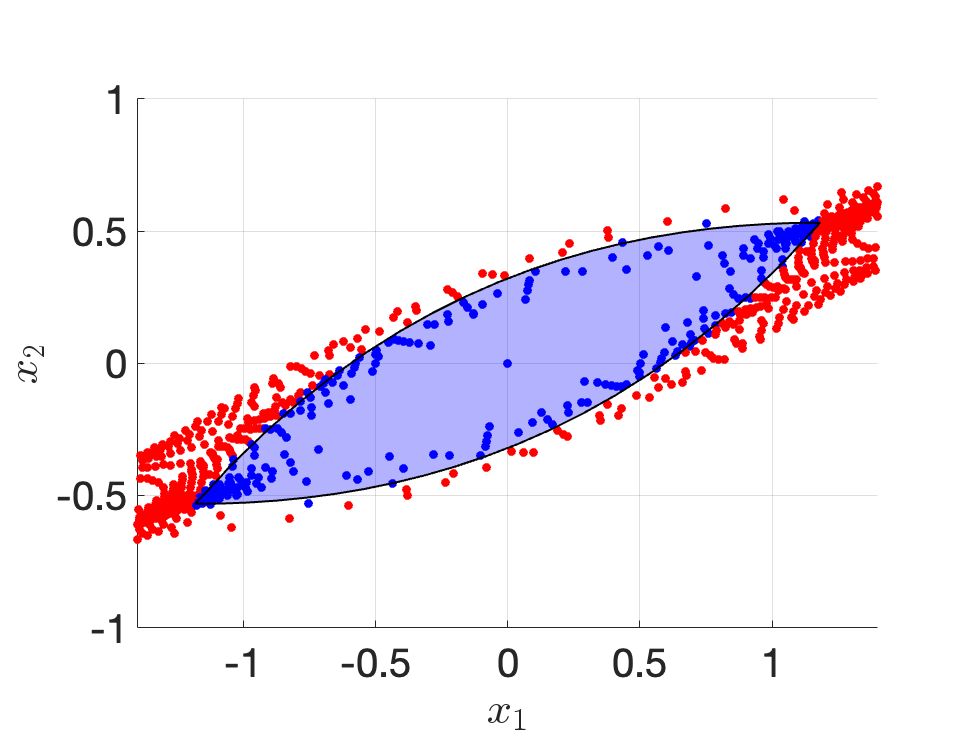}
    \caption{Numerical approximation of $\reachsets[\U_{\Lp[6]}]{1}$ for \cref{eq:example_2d_system}.
    Blue points denote states that can be reached using unit $\Lp[6]$-norm controls, while red points denote states that cannot be reached with unit $\Lp[6]$-norm controls.
    A zonotope inner-approximation of $\reachsets[\calU_{-1,1}]{1}$ is shown by the shaded region.}
    \label{fig:2d_p6_vs_cora}
\end{figure}

This approximation is shown in \cref{fig:2d_p6_vs_cora}, where the blue points denote states that can be reached using inputs with unit $\Lp[6]$-norm controls and the red points denote states where this is not the case.
A zonotopic inner-approximation of $\reachsets[\calU_{-1,1}]{1}$ is also shown for comparison in \cref{fig:2d_p6_vs_cora}.
Here, it is evident that $\reachsets[\U_{\Lp[6]}]{1}$ is a tight outer approximation of the bounded-magnitude case.

\end{example}

We note that the approach used to compute reachable sets in~\cref{example:lp_norm_reach_set} is similar to the sampling-based approach for reachability in~\cite{lew_ConvexHullsReachable_2025}.
However, our approach considers the case of linear systems and signal-norm-bounded inputs, whereas theirs considers reachability of nonlinear systems driven by a disturbance drawn from an ovaloid set.
Furthermore, they assume full actuation, whereas our approach considers the underactuated case.

\subsection{Sufficient condition for inclusion in p-norm reachable set}

While sampling costates $\blambda(0)$ allows us to approximate reachable sets as in \cref{example:lp_norm_reach_set}, imposing a norm-bound requirement on $\blambda(0)$ and searching for an inner-approximation can reduce the required computational effort.
This approach relies on the following result, which presents a sufficient condition to identify points in the $\Lp$-norm reachable set.

\begin{proposition} \label{prop:sufficient_p_norm}
    Given \cref{eq:basic_linear_system}, consider $p$ even and $\bu(t)$ in the form of \cref{eq:opt_p_control_even}.
    Let $q = \frac{p}{p-1}$.
    Then, the following bound holds:\begin{equation}
        \Lpnorm{\bu(t)}^p \leq m \left( \int\limits_0^T \|\vect(e^{-\bA \tau}\bB)\|_p^q \rd \tau \right) \|\blambda(0)\|_q^q.
    \end{equation}
\end{proposition}
\begin{proof}
    This proof follows from a straightforward computation of the $\Lp$-norm of $\bu(t)$ using the form in~\eqref{eq:opt_p_control_even} with an application of Holder's inequality.
\end{proof}

Using \cref{prop:sufficient_p_norm}, we can identify a subset of the interior of $\Lp$-norm reachable sets by imposing a bound on $\|\blambda(0)\|_q^q$, with $q$ defined as in \cref{prop:sufficient_p_norm}.
This is demonstrated in the following example.
\begin{example}
    Consider again system \cref{eq:example_2d_system} as in \cref{example:lp_norm_reach_set}.
    We wish to identify a subset of the interior of $\reachsets[\U_{\Lp[6]}]{1}$, which was approximated in \cref{example:lp_norm_reach_set}.
    Here, $q = 1.2$ and we simulate only trajectories corresponding to controls using $\blambda(0)$ satisfying $\|\blambda(0)\|_q^q~\leq~\left( m  \int_0^1 \|\vect(e^{-\bA \tau}\bB)\|_p^q \rd \tau  \right)^{-1}$.
    According to \cref{prop:sufficient_p_norm}, this ensures the control inputs will satisfy $\Lpnorm{\bu(t)} \leq 1$.
    The set of points identified by this approach is compared with the plots from \cref{example:lp_norm_reach_set} in \cref{fig:sufficient_p_norm_example}.
    \begin{figure}[!t]
        \centering
        \includegraphics[width=0.8\linewidth]{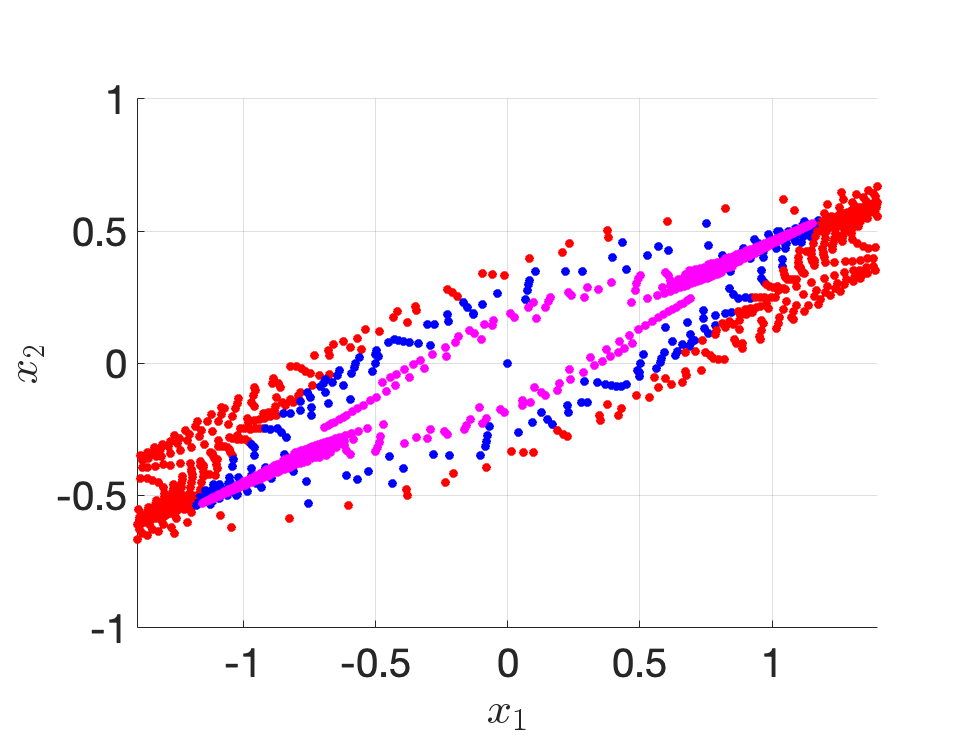}
        \caption{The magenta points show a subset of the interior of $\reachsets[\U_{\Lp[6]}]{1}$ using trajectories that satisfy a norm bound on $\blambda(0)$.
        The blue and red points are the same as in \cref{fig:2d_p6_vs_cora}.}
        \label{fig:sufficient_p_norm_example}
    \end{figure}
    Although this approach only identifies a subset of $\reachsets[\U_{\Lp[6]}]{1}$, this approach is still able to identify the farthest points along the longest axis in the reachable set, as seen in~\cref{fig:sufficient_p_norm_example}.
\end{example}

\section{Incorporating reachability into design optimization} \label{sec:example_optimizations}
In this section, we incorporate reachability analysis into a design optimization problem for a highly-maneuverable aircraft.
In the first example, we use results on $\Lp[2]$-norm reachable sets of linear systems~\cite{kalman_ContributionsTheoryOptimal_1960} and use the trace of the reachability Gramian as a constraint.
Rather than optimizing the aircraft design for a specific maneuver, we present a general formulation for a general notion of aircraft maneuverability.
In the second example, we use the procedure outlined in \cref{sec:p_norm_reach_sets} and consider the volume of the $\Lp[6]$-norm reachable set as a design constraint.
The volume of the reachable set is computed by using Scipy's implementation of QHull~\cite{barber_QuickhullAlgorithmConvex_1996}, which computes volumes of convex hulls using Delaunay triangulation.

Both examples use the same aircraft model.
The equations of motion use the following linear-longitudinal flight model from~\cite{lavretsky_RobustAdaptiveControl_2024}:  
\begin{align} \nonumber
\begin{pmatrix}
\dot{v}_T \\
\dot{\alpha} \\
\dot{q} \\
\dot{\theta}
\end{pmatrix}
=&
\begin{pmatrix}
X_V & X_\alpha & 0 & -g \cos \gamma_0 \\
\frac{Z_V}{V_0} & \frac{Z_\alpha}{V_0} & 1 + \frac{Z_q}{V_0} & -\frac{g \sin \gamma_0}{V_0} \\
M_V & M_\alpha & M_q & 0 \\
0 & 0 & 1 & 0
\end{pmatrix}
\begin{pmatrix}
v_T \\
\alpha \\
q \\
\theta
\end{pmatrix} \\ \label{eq:lavret_wise_linear_long_model}
&+
\begin{pmatrix}
X_{\delta_{\textrm{th}}} \cos \alpha_0 & X_{\delta_{\textrm{e}}} \\
- X_{\delta_{\textrm{th}}} \sin \alpha_0 & \frac{Z_{\delta_{\textrm{e}}}}{V_0} \\
M_{\delta_{\textrm{th}}} & M_{\delta_{\textrm{e}}} \\
0 & 0
\end{pmatrix}
\begin{pmatrix}
\delta_{\textrm{th}} \\
\delta_{\textrm{e}}
\end{pmatrix}.
\end{align}

For the aerodynamic model, we used the aerodynamic force and moment coefficients from~\cite{morelli_GlobalNonlinearParametric_1998}.
The coefficients in the $\bA,\bB$ matrices in \cref{eq:lavret_wise_linear_long_model} are dependent on the aircraft physical parameters and flight operating conditions, but this is omitted for brevity.

The linearization point for this model is chosen to reflect low-altitude, steady flight.
The trimmed values of angle of attack, airspeed, altitude, pitch rate, and flight path angle were set as $\alpha_0 = 12^\circ,\, V_0 = 150 \textrm{ knots},\, h_0 = 5,000 \textrm{ feet},\, q_0 = 0 \textrm{ rad/s},\, \gamma_0 = 0^\circ$ respectively.

\begin{example} \label{example:f16_L2}
    Consider the optimization problem:
    \begin{align} \nonumber
        \min_{b,\ol{c}} \quad & b + \ol{c} \\ \nonumber
        s.t. \quad & 0.5b_0 \leq b \leq 1.5 b_0 \\ \label{eq:f16_optimization_L2} 
        & 0.5 \ol{c} \leq \ol{c} \leq 1.5 \ol{c} \\ \nonumber
        & \tr (\reachgram[b,\bar{c}]{1}) \geq 1.1 \tr (\reachgram[b_0,\bar{c}_0]{1})
    \end{align}
    where $b$ and $\ol{c}$ are wingspan and wing chord.
    The reachability Gramian corresponding to the system model with design parameters $b,\bar{c}$ is denoted by $\reachgram[b,\bar{c}]{1}$.
    The initial design parameters were chosen $b_0=9.144\textrm{ m}$ and $\ol{c}_0=3.45\textrm{ m}$, as in~\cite{morelli_GlobalNonlinearParametric_1998}.
    This specific choice of objective function has the effect of minimizing the perimeter of the wing cross-section, but can be chosen as any function of the design variables.
    The first two constraints impose that the values for the wingspan and wing chord do not change by more than $50\%$ from their initial values.
    The third condition is a function of the reachability Gramian, which governs the geometry of the $\Lp[2]$-bounded-input reachable sets, and its trace is proportional to the average lengths of the principal axes of the corresponding reachable set.
    
    Solving \cref{eq:f16_optimization_L2} using sequential quadratic programming in Python's pySLSQP package~\cite{joshy_PySLSQPTransparentPython_2024} results in values $b^* = 13.62 \textrm{ m}$ and $\ol{c}^* = 3.22 \textrm{ m}$.
    Substituting this into the model and evaluating the reachability Gramian shows that $\tr(\reachgram[*]{1}) = 83.3$ is exactly a 10\% increase over $\tr(\reachgram[0]{1}) = 75.7$.
    The objective function is $16.84 \textrm{ m}$ compared to the initial value of 12.6 m.
    This means that the wing perimeter must increase by at least $33\%$ to increase the trace of the reachability Gramian by $10\%$.
\end{example}
\begin{example} \label{example:f16_L6}
    Consider the optimization problem:
    \begin{align} \nonumber
        \min_{b,\ol{c}} \quad & b + \ol{c} \\ \nonumber
        s.t. \quad & 0.5b_0 \leq b \leq 1.5 b_0 \\ \label{eq:f16_optimization_L6} 
        & 0.5 \ol{c} \leq \ol{c} \leq 1.5 \ol{c} \\ \nonumber
        & \vol( \reachsetsExp[\U_{\Lp[6]}]{1}{b,\bar{c}} ) \geq 1.1 \vol( \reachsetsExp[\U_{\Lp[6]}]{1}{b_0,\bar{c}_0}).
    \end{align}
    This optimization problem is the same as \cref{eq:f16_optimization_L2}, except the third constraint now imposes that the volume of $\reachsets[\U_{\Lp[6]}]{1}$, the $\Lp[6]$-norm reachable set corresponding to design parameters $b,\bar{c}$, is 10\% larger than at its initial volume.
    This constraint can be interpreted as increasing the volume of an approximation of the $\Lp[\infty]$ set.
    We note that since this constraint can lead to reachable sets that are more lopsided, additional constraints on eccentricity can be imposed if lopsidedness is undesirable.
    
    Once again using pySLSQP~\cite{joshy_PySLSQPTransparentPython_2024}, the solution of \cref{eq:f16_optimization_L6} is $b^* = 7.71$ m and $\bar{c}^* = 5.01$ m.
    The optimized values of wingspan and wing chord successfully increase the reachable set volume by 10\% from 0.3202 to 0.3522.
    Substituting these into the objective function shows an increase from 12.6 m to 12.7 m.
    This means that the aircraft only needs a $0.8\%$ increase in its wing perimeter to achieve a $10\%$ increase in its $\Lp[6]$ reachable set volume.
\end{example}

\section{Summary and Conclusion}
We investigated approximations and characterizations of reachable sets of linear systems under bounded-magnitude inputs.
By studying the controls that drive the system to the boundary of its reachable sets, we formulated an explicit characterization of those sets for single input, planar systems with real and distinct eigenvalues.
We also studied $\Lp$-norm reachable sets and developed a numerical method for approximating them that can be useful for approximating $\Lp[\infty]$-norm reachable sets.
Building on those results, we incorporated metrics based on reachable sets into a design optimization problem to improve aircraft maneuverability.

Generalizing our results on reachable set boundaries will be an important step in formulating reachability metrics for design optimization that are easily computable and theoretically sharp.
Moving forward, a critical next step for the work will be to extend the reachability analysis for verifying maneuvers that escape the linear regime.
Additionally, modifying the reachable sets to also account for rate-limited inputs would be valuable for capturing more realistic constraints of actuators in practice.

\bibliographystyle{IEEEtran}
\bibliography{MyReferences}

\begin{thebibliography}{10}
\providecommand{\url}[1]{#1}
\csname url@rmstyle\endcsname
\providecommand{\newblock}{\relax}
\providecommand{\bibinfo}[2]{#2}
\providecommand\BIBentrySTDinterwordspacing{\spaceskip=0pt\relax}
\providecommand\BIBentryALTinterwordstretchfactor{4}
\providecommand\BIBentryALTinterwordspacing{\spaceskip=\fontdimen2\font plus
\BIBentryALTinterwordstretchfactor\fontdimen3\font minus \fontdimen4\font\relax}
\providecommand\BIBforeignlanguage[2]{{%
\expandafter\ifx\csname l@#1\endcsname\relax
\typeout{** WARNING: IEEEtran.bst: No hyphenation pattern has been}%
\typeout{** loaded for the language `#1'. Using the pattern for}%
\typeout{** the default language instead.}%
\else
\language=\csname l@#1\endcsname
\fi
#2}}

\bibitem{martins_MultidisciplinaryDesignOptimization_2013}
J.~R. R.~A. Martins and A.~B. Lambe, ``Multidisciplinary {{Design Optimization}}: {{A Survey}} of {{Architectures}},'' \emph{AIAA Journal}, vol.~51, no.~9, pp. 2049--2075, Sept. 2013.

\bibitem{livne_IntegratedMultidisciplinarySynthesis_1990}
E.~Livne, L.~A. Schmit, and P.~P. Friedmann, ``Towards integrated multidisciplinary synthesis of actively controlled fiber composite wings,'' \emph{Journal of Aircraft}, vol.~27, no.~12, pp. 979--992, Dec. 1990.

\bibitem{grossman_IntegratedAerodynamicStructural_1988}
B.~Grossman, Z.~Gurdal, G.~J. Strauch, W.~M. Eppard, and R.~T. Haftka, ``Integrated aerodynamic/structural design of a sailplane wing,'' \emph{Journal of Aircraft}, vol.~25, no.~9, pp. 855--860, Sept. 1988.

\bibitem{manning1999large}
V.~M. Manning, ``Large-scale design of supersonic aircraft via collaborative optimization,'' Ph.D. dissertation, Stanford University, 1999.

\bibitem{garcia-sanz_ControlCoDesignEngineering_2019}
M.~Garcia-Sanz, ``Control {{Co}}-{{Design}}: {{An}} engineering game changer,'' \emph{Advanced Control for Applications}, vol.~1, no.~1, p. e18, Dec. 2019.

\bibitem{skelton_OptimalMixStructure_1992}
R.~E. Skelton and J.~H. Kim, ``The {{Optimal Mix}} of {{Structure Redesign}} and {{Active Dynamic Controllers}},'' in \emph{1992 {{American Control Conference}}}.\hskip 1em plus 0.5em minus 0.4em\relax Chicago, IL, USA: IEEE, June 1992, pp. 2775--2779.

\bibitem{gulewicz_SetBasedRobustControl_2025}
D.~Gulewicz, T.~J. Bird, H.~C. Pangborn, and N.~Jain, ``Set-{{Based Robust Control Co-Design}}: {{Application}} to a {{Hybrid Thermal Management System}},'' \emph{ASME Letters in Dynamic Systems and Control}, vol.~5, no.~3, p. 030904, July 2025.

\bibitem{bahiamonteiro_ControlRelatedMetrics_2024}
B.~Bahia~Monteiro, ``Control {{Related Metrics}} for {{Multidisciplinary Design Optimization}},'' Ph.D. dissertation, University of Michigan, 2024.

\bibitem{khot_MulticriteriaOptimizationDesign_1998}
N.~S. Khot, ``Multicriteria {{Optimization}} for {{Design}} of {{Structures}} with {{Active Control}},'' \emph{Journal of Aerospace Engineering}, vol.~11, no.~2, pp. 45--51, Apr. 1998.

\bibitem{bahiamonteiro_DesignMetricsLanding_2024}
B.~Bahia~Monteiro, I.~Kolmanovsky, and C.~E. Cesnik, ``Design {{Metrics}} for the {{Landing}} of {{Supersonic Aircraft}} under {{Stochastic Turbulence}},'' in \emph{{{AIAA SCITECH}} 2024 {{Forum}}}.\hskip 1em plus 0.5em minus 0.4em\relax Orlando, FL: {American Institute of Aeronautics and Astronautics}, Jan. 2024.

\bibitem{gupta_ControllabilityGramianControl_2020}
R.~Gupta, W.~Zhao, and R.~K. Kapania, ``Controllability {{Gramian}} as {{Control Design Objective}} in {{Aircraft Structural Design Optimization}},'' \emph{AIAA Journal}, vol.~58, no.~7, pp. 3199--3220, July 2020.

\bibitem{cunis_IntegratingNonlinearControllability_2023}
T.~Cunis, I.~Kolmanovsky, and C.~E.~S. Cesnik, ``Integrating {{Nonlinear Controllability}} into a {{Multidisciplinary Design Process}},'' \emph{Journal of Guidance, Control, and Dynamics}, vol.~46, no.~6, pp. 1026--1037, June 2023.

\bibitem{althoff_SetPropagationTechniques_2021}
M.~Althoff, G.~Frehse, and A.~Girard, ``Set {{Propagation Techniques}} for {{Reachability Analysis}},'' \emph{Annual Review of Control, Robotics, and Autonomous Systems}, vol.~4, no.~1, pp. 369--395, May 2021.

\bibitem{maler_ComputingReachableSets_2008}
O.~Maler, ``Computing {{Reachable Sets}} : {{An Introduction}},'' Tech. Rep., 2008.

\bibitem{kurzhanski_EllipsoidalTechniquesReachability_2000}
A.~B. Kurzhanski and P.~Varaiya, ``Ellipsoidal {{Techniques}} for {{Reachability Analysis}},'' in \emph{Hybrid {{Systems}}: {{Computation}} and {{Control}}}, G.~Goos, J.~Hartmanis, J.~Van~Leeuwen, N.~Lynch, and B.~H. Krogh, Eds.\hskip 1em plus 0.5em minus 0.4em\relax Berlin, Heidelberg: Springer Berlin Heidelberg, 2000, vol. 1790, pp. 202--214.

\bibitem{chutinan_ComputingPolyhedralApproximations_1998}
A.~Chutinan and B.~Krogh, ``Computing polyhedral approximations to flow pipes for dynamic systems,'' in \emph{Proceedings of the 37th {{IEEE Conference}} on {{Decision}} and {{Control}} ({{Cat}}. {{No}}.{{98CH36171}})}, vol.~2.\hskip 1em plus 0.5em minus 0.4em\relax Tampa, FL, USA: IEEE, 1998, pp. 2089--2094.

\bibitem{frehse_SpaceExScalableVerification_2011}
G.~Frehse, C.~Le~Guernic, A.~Donz{\'e}, S.~Cotton, R.~Ray, O.~Lebeltel, R.~Ripado, A.~Girard, T.~Dang, and O.~Maler, ``{{SpaceEx}}: {{Scalable Verification}} of {{Hybrid Systems}},'' in \emph{Computer {{Aided Verification}}}, G.~Gopalakrishnan and S.~Qadeer, Eds.\hskip 1em plus 0.5em minus 0.4em\relax Berlin, Heidelberg: Springer Berlin Heidelberg, 2011, vol. 6806, pp. 379--395.

\bibitem{althoff_IntroductionCORA2015_2015}
M.~Althoff, ``An {{Introduction}} to {{CORA}} 2015,'' in \emph{{{ARCH14-15}}. 1st and 2nd {{International Workshop}} on {{Applied veRification}} for {{Continuous}} and {{Hybrid Systems}}}, ser. {{EPiC Series}} in {{Computing}}, vol.~34.\hskip 1em plus 0.5em minus 0.4em\relax EasyChair, 2015, pp. 120--87.

\bibitem{bansal_HamiltonJacobiReachabilityBrief_2017}
S.~Bansal, M.~Chen, S.~Herbert, and C.~J. Tomlin, ``Hamilton-{{Jacobi}} reachability: {{A}} brief overview and recent advances,'' in \emph{2017 {{IEEE}} 56th {{Annual Conference}} on {{Decision}} and {{Control}} ({{CDC}})}.\hskip 1em plus 0.5em minus 0.4em\relax Melbourne, Australia: IEEE, Dec. 2017, pp. 2242--2253.

\bibitem{siefert_ReachabilityAnalysisUsing_2025}
J.~A. Siefert, T.~J. Bird, A.~F. Thompson, J.~J. Glunt, J.~P. Koeln, N.~Jain, and H.~C. Pangborn, ``Reachability {{Analysis Using Hybrid Zonotopes}} and {{Functional Decomposition}},'' \emph{IEEE Transactions on Automatic Control}, vol.~70, no.~7, pp. 4671--4686, July 2025.

\bibitem{bansal_DeepReachDeepLearning_2021}
S.~Bansal and C.~J. Tomlin, ``{{DeepReach}}: {{A Deep Learning Approach}} to {{High-Dimensional Reachability}},'' in \emph{2021 {{IEEE International Conference}} on {{Robotics}} and {{Automation}} ({{ICRA}})}.\hskip 1em plus 0.5em minus 0.4em\relax Xi'an, China: IEEE, May 2021, pp. 1817--1824.

\bibitem{lew_ConvexHullsReachable_2025}
T.~Lew, R.~Bonalli, and M.~Pavone, ``Convex {{Hulls}} of {{Reachable Sets}},'' \emph{IEEE Transactions on Automatic Control}, vol.~70, no.~12, pp. 8195--8209, Dec. 2025.

\bibitem{gillula_DesignGuaranteedSafe_2010}
J.~H. Gillula, {Haomiao Huang}, M.~P. Vitus, and C.~J. Tomlin, ``Design of guaranteed safe maneuvers using reachable sets: {{Autonomous}} quadrotor aerobatics in theory and practice,'' in \emph{2010 {{IEEE International Conference}} on {{Robotics}} and {{Automation}}}.\hskip 1em plus 0.5em minus 0.4em\relax Anchorage, AK: IEEE, May 2010, pp. 1649--1654.

\bibitem{porges_ReachabilityCapabilityAnalysis_2014}
O.~Porges, T.~Stouraitis, C.~Borst, and M.~A. Roa, ``Reachability and {{Capability Analysis}} for {{Manipulation Tasks}},'' in \emph{{{ROBOT2013}}: {{First Iberian Robotics Conference}}}, M.~A. Armada, A.~Sanfeliu, and M.~Ferre, Eds.\hskip 1em plus 0.5em minus 0.4em\relax Cham: Springer International Publishing, 2014, vol. 253, pp. 703--718.

\bibitem{bayen_AircraftAutolanderSafety_2007}
A.~M. Bayen, I.~M. Mitchell, M.~M.~K. Oishi, and C.~J. Tomlin, ``Aircraft {{Autolander Safety Analysis Through Optimal Control-Based Reach Set Computation}},'' \emph{Journal of Guidance, Control, and Dynamics}, vol.~30, no.~1, pp. 68--77, Jan. 2007.

\bibitem{kalman_ContributionsTheoryOptimal_1960}
R.~Kalman, ``Contributions to the theory of optimal control,'' \emph{Boletin de la Sociedad Matematica Mexicana}, vol.~5, no.~2, pp. 102--119, 1960.

\bibitem{rockafellar_ConvexAnalysis_2015}
R.~T. Rockafellar, \emph{Convex {{Analysis}}}, ser. Princeton {{Landmarks}} in {{Mathematics}} and {{Physics}}.\hskip 1em plus 0.5em minus 0.4em\relax Princeton: Princeton University Press, 2015.

\bibitem{pecsvaradi_ReachableSetsLinear_1971}
T.~Pecsvaradi and K.~S. Narendra, ``Reachable sets for linear dynamical systems,'' \emph{Information and Control}, vol.~19, no.~4, pp. 319--344, Nov. 1971.

\bibitem{pontryagin_MathematicalTheoryOptimal_2018}
L.~Pontryagin, \emph{Mathematical {{Theory}} of {{Optimal Processes}}}, 1st~ed.\hskip 1em plus 0.5em minus 0.4em\relax Routledge, May 2018.

\bibitem{athans2007optimal}
M.~Athans and P.~Falb, \emph{Optimal Control: {{An}} Introduction to the Theory and Its Applications}, ser. Dover Books on Engineering.\hskip 1em plus 0.5em minus 0.4em\relax Dover Publications, 2007.

\bibitem{barber_QuickhullAlgorithmConvex_1996}
C.~B. Barber, D.~P. Dobkin, and H.~Huhdanpaa, ``The quickhull algorithm for convex hulls,'' \emph{ACM Transactions on Mathematical Software}, vol.~22, no.~4, pp. 469--483, Dec. 1996.

\bibitem{lavretsky_RobustAdaptiveControl_2024}
E.~Lavretsky and K.~A. Wise, \emph{Robust and {{Adaptive Control}}: {{With Aerospace Applications}}}, ser. Advanced {{Textbooks}} in {{Control}} and {{Signal Processing}}.\hskip 1em plus 0.5em minus 0.4em\relax Cham: Springer International Publishing, 2024.

\bibitem{morelli_GlobalNonlinearParametric_1998}
E.~Morelli, ``Global nonlinear parametric modelling with application to {{F-16}} aerodynamics,'' in \emph{Proceedings of the 1998 {{American Control Conference}}. {{ACC}}}.\hskip 1em plus 0.5em minus 0.4em\relax Philadelphia, PA, USA: IEEE, 1998, pp. 997--1001 vol.2.

\bibitem{joshy_PySLSQPTransparentPython_2024}
A.~J. Joshy and J.~T. Hwang, ``{{PySLSQP}}: {{A}} transparent {{Python}} package for the {{SLSQPoptimization}} algorithm modernized with utilities for visualization andpost-processing,'' \emph{Journal of Open Source Software}, vol.~9, no. 103, p. 7246, Nov. 2024.

\end{thebibliography}

\end{document}